\documentclass[11pt]{article}

\usepackage{amssymb,amsmath,euscript,bbm,xcolor,graphicx,epstopdf}
\usepackage{verbatim} 
\newtheorem{proposition}{Proposition}[section]
\newtheorem{lemma}[proposition]{Lemma}
\newtheorem{theorem}[proposition]{Theorem}

\def\la{\lambda}

\def\ep{\varepsilon}

\def\l{{\langle}}
\def\r{\rangle}

\newcommand{\wt}{\widetilde}
\def\R{{\mathbb R}}

\def\E{{\mathbb E}}
\def\P{{\mathbb P}}

\makeatletter \@addtoreset{equation}{section} \makeatother
\newcommand {\qed}%
{%
    {}\hfill
    {}\hfill
    {$\square $}%
    \vspace {0.3cm}%
    \pagebreak [2]%
    \par
}%
\newenvironment{proof}[1]{%
    \vspace{0.3cm}%
    \pagebreak [2]%
    \par%
    \noindent {\bf  Proof~#1\ }}{\qed}%
\newenvironment{remark}{%
    \vspace{0.3cm} \pagebreak [2]%
    \par%
    \refstepcounter{proposition}
    \noindent%
    {\bf Remark~\theproposition\  }}{}%
\begin{document}

\title {On the Explicit Height Distribution and Expected Number of Local Maxima of Isotropic Gaussian Random Fields \thanks{Research partially
supported by NIH grant R01-CA157528.}}
\author{Dan Cheng\\ North Carolina State University
 \and Armin Schwartzman \\ North Carolina State University }

\maketitle

\begin{abstract}
The explicit formulae for the height distribution and expected number of local maxima have been obtained for isotropic Gaussian random fields on certain low-dimensional Euclidean space or low-dimensional spheres.
\end{abstract}

{\bf Keywords:} Height distribution; local maxima; Gaussian orthogonal ensemble; isotropic; Gaussian random field; sphere.

\section{Introduction}
It has been known that, by using the Kac-Rice formula, one can obtain implicit formulae for the height distribution and expected number of local maxima of smooth Gaussian random fields parameterized over an $N$-dimensional space. See, for example, Adler and Taylor (2007) and Cheng and Schwartzman (2015). In particular, if the Gaussian field is isotropic, then the application of certain Gaussian random matrices involving Gaussian orthogonal ensemble (GOE) makes the above implicit formulae computable [cf. Cheng and Schwartzman (2015) and Fyodorov(2004)]. However, for an arbitrary dimension $N$, the general explicit formulae are still hard to formulate due to the difficulty of evaluating certain GOE computation, see Theorems \ref{Thm:Palm distr iso Euclidean} and \ref{Thm:Palm distr iso sphere} below. Motivated by statistical applications of detection of peaks in Cheng and Schwartzman (2014), we investigate here the explicit formulae for the height distribution and expected number of local maxima of isotropic Gaussian random fields on $N$-dimensional Euclidean space or $N$-dimensional spheres, where $N=1, 2, 3$.

\section{Preliminary GOE Computation}
Recall that an $N\times N$ random matrix $M_N$ is said to have the Gaussian Orthogonal Ensemble (GOE) distribution if it is symmetric, with centered Gaussian entries $M_{ij}$ satisfying ${\rm Var}(M_{ii})=1$, ${\rm Var}(M_{ij})=1/2$ if $i<j$ and the random variables $\{M_{ij}, 1\leq i\leq j\leq N\}$ are independent. Moreover, the explicit formula for the distribution $Q_N$ of the eigenvalues $\la_i$ of $M_N$ is given by
\begin{equation}\label{Eq:GOE density}
Q_N(d\la)=\frac{1}{c_N} \prod_{i=1}^N e^{-\frac{1}{2}\la_i^2}d\la_i \prod_{1\leq i<j\leq N}|\la_i-\la_j|\mathbbm{1}_{\{\la_1\leq\ldots\leq\la_N\}},
\end{equation}
where the normalization constant $c_N$ can be computed from Selberg's integral
\begin{equation}\label{Eq:normalization constant}
c_N=\frac{1}{N!}(2\sqrt{2})^N \prod_{i=1}^N\Gamma\Big(1+\frac{i}{2}\Big).
\end{equation}
Denote by $\mathcal{N}(\mu,\sigma^2)$ a Gaussian distribution with mean $\mu$ and variance $\sigma^2$. Let $\Phi(x)=\P(X\le x)$ and $\Phi_\Sigma(x_1, x_2)=\P(X_1\le x_1, X_2\le x_2)$, where $X\sim \mathcal{N}(0,1)$ and $(X_1, X_2)$ is a centered bivariate Gaussian vector with covariance $\Sigma$.  For $\gamma>0$, let
\begin{equation}\label{Eq:G}
G_\gamma(x)=\int_{-\infty}^x e^{-\gamma t^2} dt.
\end{equation}
It can be seen that $G_\gamma(x) = \sqrt{\pi/\gamma}\Phi(\sqrt{2\gamma}x)$.

\begin{lemma}\label{Lemma:G}
For constants $\gamma>0$, $\sigma>0$, $\alpha>0$ and $\beta \in \R$,
\begin{equation*}
\begin{split}
\int_{-\infty}^\infty e^{-\alpha(x-\beta)^2} G_\gamma(x) dx  &= \frac{\pi}{\sqrt{\alpha\gamma}}\Phi\left(\frac{\beta\sqrt{2\alpha\gamma}}{\sqrt{\alpha+\gamma}} \right),\\
\int_{-\infty}^\infty e^{-\alpha(y-\beta)^2} \int_{-\infty}^y e^{-\sigma x^2} G_\gamma(x)dxdy &= \frac{\pi^{3/2}}{\sqrt{\alpha\sigma\gamma}} \Phi_\Sigma(0,\beta),
\end{split}
\end{equation*}
where
\begin{equation}\label{Eq:Sigma}
\begin{split}
\Sigma = \left(
\begin{array}{cc}
\frac{\sigma+\gamma}{2\sigma\gamma} & -\frac{1}{2\sigma} \\
-\frac{1}{2\sigma} &  \frac{\sigma+\alpha}{2\sigma\alpha} \\
\end{array}
\right).
\end{split}
\end{equation}
\end{lemma}
\begin{proof}\
Let $Z_1\sim \mathcal{N}(0,\frac{1}{2\gamma})$, $Z_2\sim \mathcal{N}(0,\frac{1}{2\sigma})$ and $Z_3\sim \mathcal{N}(\beta,\frac{1}{2\alpha})$ be independent. Then $(Z_1-Z_3)\sim \mathcal{N}(-\beta,\frac{\alpha+\gamma}{2\alpha\gamma})$ and hence
\begin{equation*}
\int_{-\infty}^\infty g(x) e^{-\alpha(x-\beta)^2} dx  = \frac{\pi}{\sqrt{\alpha\gamma}}\P(Z_1-Z_3 <0) = \frac{\pi}{\sqrt{\alpha\gamma}}\Phi\left(\frac{\beta\sqrt{2\alpha\gamma}}{\sqrt{\alpha+\gamma}} \right).
\end{equation*}

Let $\mu = (0, -\beta)$ and let $\Sigma$ be as defined in \eqref{Eq:Sigma}. Then $(Z_1-Z_2, Z_2-Z_3)\sim \mathcal{N}(\mu,\Sigma)$, which implies
\begin{equation*}
\begin{split}
\int_{-\infty}^\infty e^{-\alpha(y-\beta)^2} \int_{-\infty}^y e^{-x^2} g(x)dxdy &= \frac{\pi^{3/2}}{\sqrt{\alpha\sigma\gamma}} \P(Z_1-Z_2<0, Z_2-Z_3<0) \\
&= \frac{\pi^{3/2}}{\sqrt{\alpha\sigma\gamma}} \Phi_\Sigma(0,\beta).
\end{split}
\end{equation*}
\end{proof}

\begin{lemma}\label{Lemma:GOE expectation for N=1}
Let $N=1$. Then for constants $a>0$ and $b\in \R$,
\begin{equation*}
\begin{split}
\E_{GOE}^{N+1}\left\{ \exp\left[\frac{1}{2}\la_{N+1}^2 -a(\la_{N+1}-b)^2 \right] \right\} =\frac{\sqrt{4a+2}}{4a}e^{-\frac{ab^2}{2a+1}}  + \frac{b\sqrt{\pi}}{\sqrt{2a}}\Phi\left(\frac{b\sqrt{2a}}{\sqrt{2a+1}} \right).
\end{split}
\end{equation*}
\end{lemma}
\begin{proof}\ By (\ref{Eq:GOE density}) and integration by parts,
\begin{equation*}
\begin{split}
&c_{N+1}\E_{GOE}^{N+1}\left\{ \exp\left[\frac{1}{2}\la_{N+1}^2 -a(\la_{N+1}-b)^2 \right] \right\}
= \int_{-\infty}^\infty e^{-a(\la_2-b)^2} d\la_2 \int_{-\infty}^{\la_2} e^{-\frac{\la_1^2}{2}}(\la_2-\la_1) d\la_1\\
&\quad =b\int_{-\infty}^\infty G_{1/2}(\la_2) e^{-a(\la_2-b)^2} d\la_2  + \frac{\sqrt{2\pi(2a+1)}}{2a}e^{-\frac{ab^2}{2a+1}}.
\end{split}
\end{equation*}
Applying Lemma \ref{Lemma:G} and noting that $c_{N+1} = 2\sqrt{\pi}$ when $N=1$, we obtain the desired result.
\end{proof}

\begin{lemma}\label{Lemma:GOE expectation for N=2}
Let $N=2$. Then for constants $a>0$ and $b\in \R$,
\begin{equation*}
\begin{split}
&\quad \E_{GOE}^{N+1}\left\{ \exp\left[\frac{1}{2}\la_{N+1}^2 -a(\la_{N+1}-b)^2 \right] \right\} \\
&=\left(\frac{1}{a} + 2b^2-1\right)\frac{1}{\sqrt{2a}}\Phi\left(\frac{b\sqrt{2a}}{\sqrt{a+1}} \right) + \frac{b\sqrt{a+1}}{\sqrt{2\pi}a}e^{-\frac{ab^2}{a+1}} \\
&\quad + \frac{\sqrt{2}}{\sqrt{2a+1}}e^{-\frac{ab^2}{2a+1}}\Phi\left(\frac{\sqrt{2}ab}{\sqrt{(2a+1)(a+1)}} \right).
\end{split}
\end{equation*}
\end{lemma}
\begin{proof}\ By (\ref{Eq:GOE density}) and integration by parts,
\begin{equation*}
\begin{split}
&\quad c_{N+1}\E_{GOE}^{N+1}\left\{ \exp\left[\frac{1}{2}\la_{N+1}^2 -a(\la_{N+1}-b)^2 \right] \right\}\\
&=\int_{-\infty}^\infty e^{-a(\la_3-b)^2} d\la_3 \int_{-\infty}^{\la_3} e^{-\frac{\la_2^2}{2}}(\la_3-\la_2) d\la_2 \int_{-\infty}^{\la_2} e^{-\frac{\la_1^2}{2}}(\la_2-\la_1)(\la_3-\la_1) d\la_1\\
&= \int_{-\infty}^\infty e^{-a(\la_3-b)^2} \left[ (2\la_3^2-1)G_1(\la_3) + e^{-\frac{\la_3^2}{2}} G_{1/2}(\la_3) +  \la_3e^{-\la_3^2} \right]d\la_3.
\end{split}
\end{equation*}
Applying integration by parts and Lemma \ref{Lemma:G}, and noting that $c_{N+1} = \sqrt{2}\pi$ when $N=2$, we obtain the desired result.
\end{proof}

\begin{lemma}\label{Lemma:GOE expectation for N=3}
Let $N=3$. Then for constants $a>0$ and $b\in \R$,
\begin{equation*}
\begin{split}
&\quad \E_{GOE}^{N+1}\left\{ \exp\left[\frac{1}{2}\la_{N+1}^2 -a(\la_{N+1}-b)^2 \right] \right\} \\
& =\left[\frac{24a^3+12a^2+6a+1}{2a(2a+1)^2}b^2 + \frac{6a^2+3a+2}{4a^2(2a+1)} + \frac{3}{2}\right]\frac{1}{\sqrt{2(2a+1)}}e^{-\frac{ab^2}{2a+1}}\Phi\left(\frac{2\sqrt{2}ab}{\sqrt{(2a+1)(2a+3)}} \right)\\
& \quad+ \left[\frac{a+1}{2a}b^2 + \frac{1-a}{2a^2} -1\right] \frac{1}{\sqrt{2(a+1)}} e^{-\frac{ab^2}{a+1}}\Phi\left(\frac{\sqrt{2}ab}{\sqrt{(a+1)(2a+3)}} \right)\\
&\quad+\left(6a+1+\frac{28a^2+12a+3}{2a(2a+1)}\right)\frac{b}{2\sqrt{2\pi}(2a+1)\sqrt{2a+3}}e^{-\frac{3ab^2}{2a+3}}\\
& \quad+ \left[ b^2+\frac{3(1-a)}{2a}\right]\frac{\sqrt{\pi}b}{\sqrt{2a}} \left[\Phi_{\Sigma_1}(0,b) + \Phi_{\Sigma_2}(0,b)\right],
\end{split}
\end{equation*}
where
\begin{equation*}
\begin{split}
\Sigma_1 = \left(
\begin{array}{cc}
\frac{3}{2} & -\frac{1}{2} \\
-\frac{1}{2} &  \frac{1+a}{2a} \\
\end{array}
\right), \quad  \Sigma_2 = \left(
\begin{array}{cc}
\frac{3}{2} & -1 \\
-1 & \frac{1+2a}{2a} \\
\end{array}
\right).
\end{split}
\end{equation*}
\end{lemma}
\begin{proof}\ By (\ref{Eq:GOE density}) and integration by parts,
\begin{equation*}
\begin{split}
&\quad c_{N+1}\E_{GOE}^{N+1}\left\{ \exp\left[\frac{1}{2}\la_{N+1}^2 -a(\la_{N+1}-b)^2 \right] \right\}\\
&=\int_{-\infty}^\infty e^{-a(\la_4-b)^2} d\la_4 \int_{-\infty}^{\la_4} e^{-\frac{\la_3^2}{2}}(\la_4-\la_3) d\la_3
\int_{-\infty}^{\la_3} e^{-\frac{\la_2^2}{2}}(\la_4-\la_2)(\la_3-\la_2) d\la_2 \\
&\quad \int_{-\infty}^{\la_2} e^{-\frac{\la_1^2}{2}}(\la_2-\la_1)(\la_3-\la_1)(\la_4-\la_1) d\la_1\\
&=\left[(3-\frac{1}{4a^2})\frac{2a+1+4a^2b^2}{(2a+1)^2} + (\frac{b}{4a^2} + \frac{3b}{2a})\frac{2ab}{2a+1} + \frac{3-3a}{4a^2} + \frac{3}{2}\right]\\
&\quad \times e^{-\frac{ab^2}{2a+1}}\int_{-\infty}^\infty e^{-\frac{2a+1}{2}(\la_4-\frac{2ab}{2a+1})^2} G_1(\la_4)d\la_4\\
& \quad+ \left[\frac{3b}{2a}+ b^3-\frac{3}{2}b\right]\int_{-\infty}^\infty e^{-a(\la_4-b)^2} \int_{-\infty}^{\la_4} e^{-\frac{\la_3^2}{2}} G_1(\la_3)d\la_3d\la_4\\
& \quad+ \left[(\frac{1}{2}-\frac{1}{2a^2})\frac{a+1+2a^2b^2}{2(a+1)^2} + (\frac{b}{2a^2} + \frac{3b}{2a})\frac{ab}{a+1} + \frac{3-3a}{4a^2} -1\right]\\
&\quad \times e^{-\frac{ab^2}{a+1}}\int_{-\infty}^\infty  e^{-(a+1)(\la_4-\frac{ab}{a+1})^2} G_{1/2}(\la_4)d\la_4\\
& \quad+ \left[\frac{3b}{2a}+ b^3-\frac{3}{2}b\right]\int_{-\infty}^\infty e^{-a(\la_4-b)^2}  \int_{-\infty}^{\la_4} e^{-\la_3^2}G_{1/2}(\la_3)d\la_3d\la_4\\
& \quad+ \int_{-\infty}^\infty  \left[\frac{(6a+1)(2a+3)}{4a(2a+1)}\la_4 + \frac{(28a^2+12a+3)b}{4a(2a+1)^2}\right]e^{-a(\la_4-b)^2}e^{-\frac{3\la_4^2}{2}}d\la_4.
\end{split}
\end{equation*}
Applying integration by parts and Lemma \ref{Lemma:G}, and noting that $c_{N+1} = 2\pi$ when $N=3$, we obtain the desired result.
\end{proof}

\section{Isotropic Gaussian Random Fields on Euclidean Space}
In this section, we consider smooth isotropic Gaussian fields on Euclidean space. Let $\{f(t): t\in \R^N\}$ be a real-valued, $C^2$, centered, unit-variance isotropic Gaussian field. Due to isotropy, we can write the covariance function of the field as $\E\{f(t)f(s)\}=\rho(\|t-s\|^2)$ for an appropriate function $\rho(\cdot): [0,\infty) \rightarrow \R$, and denote
\begin{equation}\label{Eq:kappa}
\rho'=\rho'(0), \quad \rho''=\rho''(0), \quad \kappa=-\rho'/\sqrt{\rho''}.
\end{equation}
Let $f_i (t)=\frac{\partial f(t)}{\partial t_i}$ and $f_{ij}(t)=\frac{\partial^2 f(t)}{\partial t_i\partial t_j}$. Denote by $\nabla f(t)$ and $\nabla^2 f(t)$ the column vector $(f_1(t), \ldots , f_N(t))^{T}$ and the $N\times N$ matrix $(f_{ij}(t))_{ i, j = 1, \ldots, N}$, respectively. By isotropy again, the covariance of $(f(t), \nabla f(t), \nabla^2 f(t))$ only depends on $\rho'$ and $\rho''$. In particular, ${\rm Var}(f_i(t))=-2\rho'$ and ${\rm Var}(f_{ii}(t))=12\rho''$ for any $i\in\{1,\ldots, N\}$, which implies $\rho'<0$ and $\rho''>0$ and hence $\kappa>0$.

Over the unit open cube $(0,1)^N$ in $\R^N$, define respectively the expected number of local maxima of $f$ and the expected number of local maxima of $f$ exceeding level $u$ as
\begin{equation*}
\begin{split}
M(f) & = \# \left\{ t \in (0,1)^N: \nabla f(t)=0, \text{index} (\nabla^2 f(t))=N \right\},\\
M_u(f) & = \# \left\{ t \in (0,1)^N: f(t)\geq u, \nabla f(t)=0, \text{index} (\nabla^2 f(t))=N \right\},
\end{split}
\end{equation*}
where the index of a matrix is regarded as the number of its negative eigenvalues. Define the height distribution of a local maximum of $f$ at some point, say the origin, as
\begin{equation*}
F(u) :=\lim_{\ep\to 0} \P\left\{f(0)>u | \exists \text{ a local maximum of } f(t) \text{ in } (-\ep, \ep)^N \right\}.
\end{equation*}
Then we have the following result in Cheng and Schwartzman (2015).
\begin{theorem}\label{Thm:Palm distr iso Euclidean} Let $\{f(t): t\in \R^N\}$ be a centered, unit-variance, isotropic Gaussian random field satisfying the conditions $({\bf C}1)$, $({\bf C}2)$ and $({\bf C}3)$ in Cheng and Schwartzman (2015). Then for each $u\in \R$,
\begin{equation*}
\begin{split}
\E\{M(f)\} &=\left(\frac{2}{\pi}\right)^{(N+1)/2}\Gamma\left(\frac{N+1}{2}\right)\left(-\frac{\rho''}{\rho'}\right)^{N/2}\E_{GOE}^{N+1}\left\{ \exp\left[-\frac{\la_{N+1}^2}{2} \right] \right\}
\end{split}
\end{equation*}
and
\begin{equation}\label{Eq:F-Euclidean}
\begin{split}
F(u) &= \frac{\E\{M_u(f)\}}{\E\{M(f)\}}\\
&= \frac{(1-\kappa^2)^{-1/2} \int_u^\infty \phi(x)\E_{GOE}^{N+1}\left\{ \exp\left[\la_{N+1}^2/2 - (\la_{N+1}-\kappa x/\sqrt{2} )^2/(1-\kappa^2) \right]\right\}dx}{\E_{GOE}^{N+1}\left\{ \exp\left[-\la_{N+1}^2/2 \right] \right\}},
\end{split}
\end{equation}
where $\phi(x)$ is the density of standard Gaussian variable and $\rho'$, $\rho''$ and $\kappa$ are defined in \eqref{Eq:kappa}.
\end{theorem}

\begin{remark}\label{Remark:kappa}
Recall the condition $({\bf C}3)$ in Cheng and Schwartzman (2015) that $\kappa\le 1$. We show below that, the factors $(1-\kappa^2)^{-1/2}$ and $(1-\kappa^2)^{-1}$ in \eqref{Eq:F-Euclidean} will be cancelled out, implying that \eqref{Eq:F-Euclidean} also holds when $\kappa=1$.  In fact, by \eqref{Eq:GOE density},
\begin{equation*}
\begin{split}
&\quad(1-\kappa^2)^{-1/2} \E_{GOE}^{N+1}\left\{ \exp\left[\la_{N+1}^2/2 - (\la_{N+1}-\kappa x/\sqrt{2} )^2/(1-\kappa^2) \right]\right\}\\
&= \frac{1}{c_{N+1}\sqrt{1-\kappa^2}} \int_{\R^{N+1}}e^{-\frac{(\la_{N+1}-\kappa x/\sqrt{2} )^2}{1-\kappa^2}}\prod_{i=1}^N e^{-\frac{1}{2}\la_i^2}d\la_i \prod_{1\leq i<j\leq N}|\la_i-\la_j|\\
&\quad \times \prod_{1\leq k \leq N}|\la_k-\la_{N+1}|\mathbbm{1}_{\{\la_1\leq\ldots\leq\la_N \le \la_{N+1}\}} d\la_1\cdots d\la_N d\la_{N+1}\\
&=\frac{1}{c_{N+1}} \int_{\R^{N+1}}e^{-\wt{\la}_{N+1}^2}\prod_{i=1}^N e^{-\frac{1}{2}\la_i^2}d\la_i \prod_{1\leq i<j\leq N}|\la_i-\la_j|\prod_{1\leq k \leq N}|\la_k-\sqrt{1-\kappa^2}\wt{\la}_{N+1}-\kappa x/\sqrt{2}|\\
&\quad \times \mathbbm{1}_{\{\la_1\leq\ldots\leq\la_N \le \sqrt{1-\kappa^2}\wt{\la}_{N+1}+\kappa x/\sqrt{2}\}} d\la_1\cdots d\la_Nd\wt{\la}_{N+1},
\end{split}
\end{equation*}
where we have made change of variables $\wt{\la}_{N+1}=(\la_{N+1}-\kappa x/\sqrt{2} )/\sqrt{1-\kappa^2}$.
\end{remark}

From now on, we denote by $h$ the density of the height distribution of local maxima of $f$, i.e. $h(x)=-F'(x)$ or $F(x)=\int_x^\infty h(t)dt$. Also notice that, due to the formula $\E\{M_u(f)\}=F(u)\E\{M(f)\}$ in \eqref{Eq:F-Euclidean}, we only need to find $F(u)$ and $\E\{M(f)\}$.

\begin{proposition}\label{Prop:Euclidean}
\ Let the assumptions in Theorem \ref{Thm:Palm distr iso Euclidean} hold. If $N=1$, then
\begin{equation*}
\E\{M(f)\} = \frac{\sqrt{6}}{2\pi}\sqrt{-\frac{\rho''}{\rho'}}
\end{equation*}
and
\begin{equation*}
\begin{split}
h(x)&=\frac{\sqrt{3-\kappa^2}}{\sqrt{6\pi}} e^{-\frac{3x^2}{2(3-\kappa^2)}} + \frac{2\kappa x\sqrt{\pi}}{\sqrt{6}}\phi(x)\Phi\left(\frac{\kappa x}{\sqrt{3-\kappa^2}} \right).
\end{split}
\end{equation*}
If $N=2$, then
\begin{equation*}
\E\{M(f)\} = -\frac{\rho''}{\sqrt{3}\pi\rho'}
\end{equation*}
and
\begin{equation*}
\begin{split}
h(x) &=\sqrt{3}\kappa^2(x^2-1)\phi(x)\Phi\left(\frac{\kappa x}{\sqrt{2-\kappa^2}} \right) + \frac{\kappa x\sqrt{3(2-\kappa^2)}}{2\pi}e^{-\frac{x^2}{2-\kappa^2}} \\
&\quad+\frac{\sqrt{6}}{\sqrt{\pi(3-\kappa^2)}}e^{-\frac{3x^2}{2(3-\kappa^2)}}\Phi\left(\frac{\kappa x}{\sqrt{(3-\kappa^2)(2-\kappa^2)}} \right).
\end{split}
\end{equation*}
If $N=3$, then
\begin{equation*}
\E\{M(f)\} = \frac{29\sqrt{6}-36}{36\pi^2}\left(-\frac{\rho''}{\rho'}\right)^{3/2}
\end{equation*}
and
\begin{equation*}
\begin{split}
h(x) &=\frac{144\phi(x)}{29\sqrt{6}-36}\Bigg\{\Bigg[\frac{\kappa^2\left[(1-\kappa^2)^3+6(1-\kappa^2)^2+12(1-\kappa^2)+24\right]}
{4(3-\kappa^2)^2}x^2 \\
&\quad+ \frac{2(1-\kappa^2)^3+3(1-\kappa^2)^2+6(1-\kappa^2)}{4(3-\kappa^2)} + \frac{3}{2}\Bigg] \frac{e^{-\frac{\kappa^2x^2}{2(3-\kappa^2)}}}{\sqrt{2(3-\kappa^2)}}\Phi\left(\frac{2\kappa x}{\sqrt{(3-\kappa^2)(5-3\kappa^2)}}\right)\\
&\quad + \left[\frac{\kappa^2(2-\kappa^2)}{4}x^2 - \frac{\kappa^2(1-\kappa^2)}{2} -1\right] \frac{e^{-\frac{\kappa^2x^2}{2(2-\kappa^2)}}}{\sqrt{2(2-\kappa^2)}}\Phi\left(\frac{\kappa x} {\sqrt{(2-\kappa^2)(5-3\kappa^2)}}\right)\\
&\quad + \left[7-\kappa^2 + \frac{(1-\kappa^2)\left[3(1-\kappa^2)^2 + 12(1-\kappa^2) + 28\right]}{2(3-\kappa^2)} \right] \frac{\kappa xe^{-\frac{3\kappa^2x^2}{2(5-3\kappa^2)}}}{4\sqrt{\pi}(3-\kappa^2)\sqrt{5-3\kappa^2}}\\
&\quad + \frac{\sqrt{\pi}\kappa^3}{4}x(x^2 - 3)\left[\Phi_{\Sigma_1}(0,\kappa x/\sqrt{2}) + \Phi_{\Sigma_2}(0,\kappa x/\sqrt{2})\right]\Bigg\},
\end{split}
\end{equation*}
where
\begin{equation*}
\begin{split}
\Sigma_1 = \left(
\begin{array}{cc}
\frac{3}{2} & -1 \\
-1 & \frac{3-\kappa^2}{2} \\
\end{array}
\right), \quad \Sigma_2 = \left(
\begin{array}{cc}
\frac{3}{2} & -\frac{1}{2} \\
-\frac{1}{2} & \frac{2-\kappa^2}{2} \\
\end{array}
\right).
\end{split}
\end{equation*}
\end{proposition}

\begin{proof}\
Let $N=1$. By Theorem \ref{Thm:Palm distr iso Euclidean}, we see that $\E\{M(f)\}$ can be obtained by applying Lemma \ref{Lemma:GOE expectation for N=1} with $a=1$ and $b=0$, while $h(x)$ can be obtained by applying Lemma \ref{Lemma:GOE expectation for N=1} with $a=(1-\kappa^2)^{-1}$ and $b=\kappa x/\sqrt{2}$. The cases for $N=2$ and $N=3$ can be proved similarly by applying Lemmas \ref{Lemma:GOE expectation for N=2} and \ref{Lemma:GOE expectation for N=3} respectively.
\end{proof}

\begin{figure}[h!]
  \centering
\includegraphics[scale=0.40]{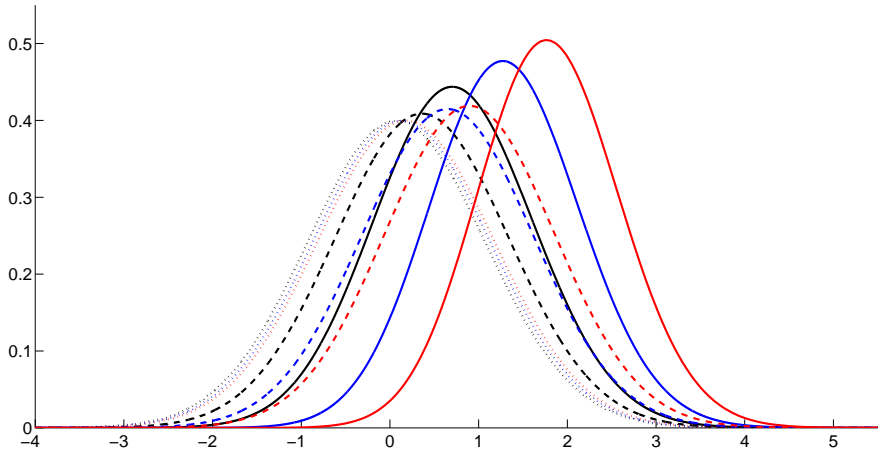}
\caption{Density functions of the height distribution of local maxima on Euclidean space. All three cases $N=1$ (black), $N=2$ (blue) and $N=3$ (red) are displayed. Solid lines represent $\kappa=1$, dashed lines represent $\kappa=0.5$ and dotted lines represent $\kappa=0.1$.}
\label{Fig:Euclidean}
\end{figure}

\begin{remark}
The regularity condition $({\bf C}2)$ in Cheng and Schwartzman (2015) implies that $\kappa^2<3$ when $N=1$, $\kappa^2<2$ when $N=2$ and $\kappa^2<5/3$ when $N=3$. Although Theorem \ref{Thm:Palm distr iso Euclidean} holds for $\kappa\le1$ as shown in Cheng and Schwartzman (2015), we conjecture that the formulae obtained in Proposition \ref{Prop:Euclidean} are valid for any isotropic Gaussian fields satisfying conditions $({\bf C}1)$, $({\bf C}2)$ and $({\bf C}3)$ in Cheng and Schwartzman (2015).
\end{remark}

\section{Isotropic Gaussian Random Fields on Sphere}
In this section, we denote by $\mathbb{S}^N$ the $N$-dimensional unit sphere and let $f= \{f(t): t\in \mathbb{S}^N\}$ be a real-valued, $C^2$, centered, unit-variance isotropic Gaussian random field on $\mathbb{S}^N$. Due to isotropy, we may write the covariance function of $f$ as $C(\l t, s\r)$, $t,s\in \mathbb{S}^N$, where $C(\cdot): [-1,1] \rightarrow \R$ and $\l \cdot, \cdot \r$ is the inner product in $\R^{N+1}$. See Cheng and Schwartzman (2015) and Cheng and Xiao (2014) for more results on the covariance function of an isotropic Gaussian field on $\mathbb{S}^N$. We define
\begin{equation}\label{Def:C' and C''}
C'=C'(1), \quad C''=C''(1), \quad \kappa_1=C'/C'', \quad \kappa_2=C'^2/C''.
\end{equation}
It is known from Cheng and Schwartzman (2015) that $C'>0$ and $C''>0$.

Let $B(t_0, \delta)$ be a geodesic (open) ball on $\mathbb{S}^N$ with radius $\delta$ centered at $t_0\in \mathbb{S}^N$. Let $\delta_0$ be the number such that the area of $B(t_0, \delta_0)$ equals 1. Following the notation in Cheng and Schwartzman (2015), over the unit geodesic ball $B(t_0, \delta_0)$ in $\mathbb{S}^N$, define respectively the expected number of local maxima of $f$ and the expected number of local maxima of $f$ exceeding level $u$ as
\begin{equation*}
\begin{split}
M(f) & = \# \left\{ t \in B(t_0, \delta_0): \nabla f(t)=0, \text{index} (\nabla^2 f(t))=N \right\},\\
M_u(f) & = \# \left\{ t \in B(t_0, \delta_0): f(t)\geq u, \nabla f(t)=0, \text{index} (\nabla^2 f(t))=N \right\}.
\end{split}
\end{equation*}
Define the height distribution of a local maximum of $f$ at some point, say $t_0$, as
\begin{equation*}
F(u) :=\lim_{\ep\to 0} \P\left\{f(t_0)>u | \exists \text{ a local maximum of } f(t) \text{ in } B(t_0, \ep) \right\}.
\end{equation*}
Then we have the following result in Cheng and Schwartzman (2015).
\begin{theorem}\label{Thm:Palm distr iso sphere} Let $\{f(t): t\in \mathbb{S}^N\}$ be a centered, unit-variance, isotropic Gaussian random field satisfying the conditions $({\bf C}1')$, $({\bf C}2')$ and $({\bf C}3')$ in Cheng and Schwartzman (2015). Then for each $u\in \R$,
\begin{equation*}
\begin{split}
\E\{M(f)\} &=\frac{\sqrt{2}}{\pi^{(N+1)/2}\kappa_1^{N/2}\sqrt{1+\kappa_1}} \Gamma\left(\frac{N+1}{2}\right) \E_{GOE}^{N+1}\left\{ \exp\left[\frac{\la_{N+1}^2}{2} -\frac{\la_{N+1}^2}{1+\kappa_1} \right] \right\}
\end{split}
\end{equation*}
and
\begin{equation}\label{Eq:F-sphere}
\begin{split}
F(u) &= \frac{\E\{M_u(f)\}}{\E\{M(f)\}}\\
&=\frac{\left(\frac{1+\kappa_1}{1+\kappa_1-\kappa_2}\right)^{1/2}\int_u^\infty \phi(x)\E_{GOE}^{N+1}\left\{ \exp\left[\frac{\la_{N+1}^2}{2} - \frac{\big(\la_{N+1}-\sqrt{\kappa_2}x/\sqrt{2} \big)^2}{1+\kappa_1-\kappa_2} \right]\right\}dx}{\E_{GOE}^{N+1}\left\{ \exp\left[\frac{\la_{N+1}^2}{2} -\frac{\la_{N+1}^2}{1+\kappa_1} \right] \right\}},
\end{split}
\end{equation}
where $\kappa_1$ and $\kappa_2$ are defined in \eqref{Def:C' and C''}.
\end{theorem}

\begin{remark}
Recall the condition $({\bf C}3')$ in Cheng and Schwartzman (2015) that $\kappa_2-\kappa_1\le 1$. Similarly to Remark \ref{Remark:kappa}, it can be shown that, the factors $(1+\kappa_1-\kappa_2)^{-1/2}$ and $(1+\kappa_1-\kappa_2)^{-1}$ in \eqref{Eq:F-sphere} will be cancelled out, implying that \eqref{Eq:F-sphere} also holds when $\kappa_2-\kappa_1= 1$.
\end{remark}

\begin{figure}[h!]
  \centering
\includegraphics[scale=0.40]{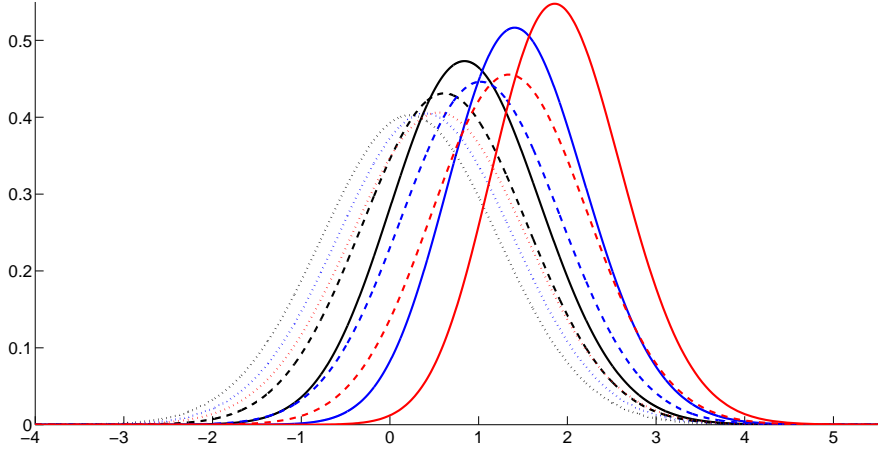}
\caption{Density functions of the height distribution of local maxima on sphere. All three cases $N=1$ (black), $N=2$ (blue) and $N=3$ (red) are displayed. Solid lines represent $\kappa_1=1$ and $\kappa_2=2$, dashed lines represent $\kappa_1=1$ and $\kappa_2=1$ and dotted lines represent $\kappa_1=0.1$ and $\kappa_2=0.1$.}
\label{Fig:Sphere}
\end{figure}

Now, we can obtain the following explicit formulae for $\E\{M(f)\}$ and the height density $h$, which imply $F$ and $\E\{M_u(f)\}$.
\begin{proposition}\label{Prop:Sphere}
\ Let the assumptions in Theorem \ref{Thm:Palm distr iso sphere} hold. If $N=1$, then
\begin{equation*}
\E\{M(f)\} = \frac{\sqrt{3+\kappa_1}}{2\pi\sqrt{\kappa_1}}
\end{equation*}
and
\begin{equation*}
\begin{split}
h(x)&= \frac{1}{\sqrt{3+\kappa_1}}\left[\frac{\sqrt{3+\kappa_1-\kappa_2}}{\sqrt{2\pi}}e^{-\frac{(3+\kappa_1) x^2}{2(3+\kappa_1-\kappa_2)}} + \sqrt{2\pi \kappa_2}x\phi(x)\Phi\left(\frac{\sqrt{\kappa_2}x}{\sqrt{3+\kappa_1-\kappa_2}}\right)\right].
\end{split}
\end{equation*}
If $N=2$, then
\begin{equation*}
\E\{M(f)\} = \frac{1}{4\pi} + \frac{1}{2\pi \kappa_1\sqrt{3+\kappa_1}}
\end{equation*}
and
\begin{equation*}
\begin{split}
h(x)&=\frac{2\sqrt{3+\kappa_1}}{2+\kappa_1\sqrt{3+\kappa_1}}\Bigg\{ \left[\kappa_1+\kappa_2(x^2-1)\right]\phi(x)\Phi\left(\frac{\sqrt{\kappa_2} x}{\sqrt{2+\kappa_1-\kappa_2}} \right) \\
&\quad+ \frac{\sqrt{\kappa_2(2+\kappa_1-\kappa_2)}}{2\pi}xe^{-\frac{(2+\kappa_1)x^2}{2(2+\kappa_1-\kappa_2)}} \\
&\quad + \frac{\sqrt{2}}{\sqrt{\pi(3+\kappa_1-\kappa_2)}}e^{-\frac{(3+\kappa_1)x^2}{2(3+\kappa_1-\kappa_2)}} \Phi\left(\frac{\sqrt{\kappa_2}x}{\sqrt{(2+\kappa_1-\kappa_2)(3+\kappa_1-\kappa_2)}} \right) \Bigg\},
\end{split}
\end{equation*}
If $N=3$, then
\begin{equation*}
\E\{M(f)\} = \frac{1}{\pi^2\kappa_1^{3/2}}\left\{ \frac{1}{2\sqrt{3+\kappa_1}}\left[ \frac{3}{2} + \frac{(1+\kappa_1)(2\kappa_1^2 + 7\kappa_1 + 11)}{4(3+\kappa_1)} \right] + \frac{(\kappa_1-1)(\kappa_1 +2)}{4\sqrt{2+\kappa_1}} \right\}
\end{equation*}
and
\begin{equation*}
\begin{split}
h(x) &=\left\{ \frac{1}{2\sqrt{2(3+\kappa_1)}}\left[ \frac{3}{2} + \frac{(1+\kappa_1)(2\kappa_1^2 + 7\kappa_1 + 11)}{4(3+\kappa_1)} \right] + \frac{(\kappa_1-1)(\kappa_1 +2)}{4\sqrt{2(2+\kappa_1)}} \right\}^{-1}\phi(x)\\
&\times \Bigg\{ \Bigg[\frac{\kappa_2\left[(1+\kappa_1-\kappa_2)^3+6(1+\kappa_1-\kappa_2)^2+12(1+\kappa_1-\kappa_2)+24\right]}{4(3+\kappa_1-\kappa_2)^2}x^2\\
&\quad + \frac{2(1+\kappa_1-\kappa_2)^3 + 3(1+\kappa_1-\kappa_2)^2 + 6(1+\kappa_1-\kappa_2)}{4(3+\kappa_1-\kappa_2)} + \frac{3}{2}\Bigg]\frac{1}{\sqrt{2(3+\kappa_1-\kappa_2)}}\\
&\quad \times e^{-\frac{\kappa_2 x^2}{2(3+\kappa_1-\kappa_2)}}\Phi\left( \frac{2\sqrt{\kappa_2}x}{\sqrt{(3+\kappa_1-\kappa_2)(5+3\kappa_1-3\kappa_2)}} \right)\\
&\quad + \left[ \frac{\kappa_2(2+\kappa_1-\kappa_2)}{4}x^2 + \frac{(\kappa_1-\kappa_2)(1+\kappa_1-\kappa_2)}{2} -1 \right]\frac{1}{\sqrt{2(2+\kappa_1-\kappa_2)}}\\
&\quad \times e^{-\frac{\kappa_2 x^2}{2(2+\kappa_1-\kappa_2)}}\Phi\left( \frac{\sqrt{\kappa_2}x}{\sqrt{(2+\kappa_1-\kappa_2)(5+3\kappa_1-3\kappa_2)}} \right)\\
&\quad + \left[ 7+\kappa_1-\kappa_2 + \frac{3(1+\kappa_1-\kappa_2)^3 + 12(1+\kappa_1-\kappa_2)^2 + 28(1+\kappa_1-\kappa_2)}{2(3+\kappa_1-\kappa_2)} \right]\\
&\quad \times \frac{\sqrt{\kappa_2}}{4\sqrt{\pi}(3+\kappa_1-\kappa_2)\sqrt{5+3\kappa_1-3\kappa_2}}x e^{-\frac{3\kappa_2 x^2}{2(5+3\kappa_1-3\kappa_2)}}\\
&\quad +\left[ \kappa_2 x^2+3(\kappa_1-\kappa_2) \right]\frac{\sqrt{\pi\kappa_2}}{4}x\left[\Phi_{\Sigma_1}\left(0,\frac{\sqrt{\kappa_2} x}{\sqrt{2}}\right) + \Phi_{\Sigma_2}\left(0,\frac{\sqrt{\kappa_2} x}{\sqrt{2}}\right)\right]\Bigg\},
\end{split}
\end{equation*}
where
\begin{equation*}
\begin{split}
\Sigma_1 = \left(
\begin{array}{cc}
\frac{3}{2} & -1 \\
-1 & \frac{3+\kappa_1-\kappa_2}{2} \\
\end{array}
\right), \quad \Sigma_2 = \left(
\begin{array}{cc}
\frac{3}{2} & -\frac{1}{2} \\
-\frac{1}{2} & \frac{2+\kappa_1-\kappa_2}{2} \\
\end{array}
\right).
\end{split}
\end{equation*}
\end{proposition}
\begin{proof}\
Let $N=1$. By Theorem \ref{Thm:Palm distr iso sphere}, we see that $\E\{M(f)\}$ can be obtained by applying Lemma \ref{Lemma:GOE expectation for N=1} with $a=(1+\kappa_1)^{-1}$ and $b=0$, while $h(x)$ can be obtained by applying Lemma \ref{Lemma:GOE expectation for N=1} with $a=(1+\kappa_1-\kappa_2)^{-1}$ and $b=\sqrt{\kappa_2} x/\sqrt{2}$. The cases for $N=2$ and $N=3$ can be proved similarly by applying Lemmas \ref{Lemma:GOE expectation for N=2} and \ref{Lemma:GOE expectation for N=3} respectively.
\end{proof}

\begin{remark}
The regularity condition $({\bf C}2')$ in Cheng and Schwartzman (2015) implies that $\kappa_2-\kappa_1<3$ when $N=1$, $\kappa_2-\kappa_1<2$ when $N=2$ and $\kappa_2-\kappa_1<5/3$ when $N=3$. Although Theorem \ref{Thm:Palm distr iso sphere} holds for $\kappa_2-\kappa_1\le1$ as shown in Cheng and Schwartzman (2015), we conjecture that the formulae obtained in Proposition \ref{Prop:Sphere} are valid for any isotropic Gaussian fields satisfying conditions $({\bf C}1')$, $({\bf C}2')$ and $({\bf C}3')$ in Cheng and Schwartzman (2015).
\end{remark}

\bibliographystyle{plain}

\begin{thebibliography}{1234}
\bibitem{Adler81}
Adler, R. J. (1981), {\it The Geometry of Random Fields}. Wiley, New York.

\bibitem{AT07}
Adler, R. J. and Taylor, J. E. (2007), {\it Random Fields and Geometry}. Springer,
New York.

\bibitem{AzaisW08}
Aza\"is, J.-M. and Wschebor, M. (2008), A general expression for the distribution of the maximum of a
Gaussian field and the approximation of the tail. {\it Stoch. Process. Appl.},
{\bf 118},  1190--1218.

\bibitem{AzaisW10}
Aza\"is, J.-M. and Wschebor, M. (2010), Erratum to: A general expression for the distribution of
the maximum of a Gaussian field and the approximation of the tail [Stochastic Process. Appl. 118 (7) (2008) 1190--1218]. {\it Stoch. Process. Appl.},
{\bf 120},  2100--2101.


\bibitem{ChengXiao14}
Cheng, D. and Xiao, Y. (2014), Excursion probability of Gaussian random fields on sphere. {\it Bernoulli,} to appear. http://arxiv.org/abs/1401.5498

\bibitem{CS15}
Cheng, D. and Schwartzman, A. (2015). Distribution of the height of local maxima of Gaussian random fields. {\it Extremes,} to appear. http://arxiv.org/abs/1307.5863

\bibitem{CS14}
Cheng, D. and Schwartzman, A. (2014), Multiple testing of local maxima for detection of peaks in random fields. arXiv:1405.1400

\bibitem{CL67}
Cram\'er, H. and Leadbetter, M. R. (1967), {\it Stationary and Related Stochastic Processes:
Sample Function Properties and Their Applications}. Wiley, New York.

\bibitem{Fyodorov04}
Fyodorov, Y. V. (2004), Complexity of random energy landscapes, glass transition, and absolute value of the spectral determinant of random matrices. {\it Phys. Rev. Lett.}, {\bf 92}, 240601.

\bibitem{Schoenberg42}
Schoenberg, I. J. (1942), Positive definite functions on spheres. {\it Duke Math. J.},
{\bf 9},  96--108.

\bibitem{SGA11}
Schwartzman, A., Gavrilov, Y. and Adler, R. J. (2011), Multiple testing of local maxima for detection of peaks in 1D. {\it Ann. Statist.}, {\bf 39}, 3290--3319.

\end{thebibliography}

\begin{small}

\end{small}

\bigskip

\begin{quote}
\begin{small}

\textsc{Dan Cheng}: Department of Statistics, North Carolina State
University, 2311 Stinson Drive, Campus Box 8203, Raleigh, NC 27695, U.S.A.\\
E-mail: \texttt{dcheng2@ncsu.edu}\\
[1mm]

\textsc{Armin Schwartzman}: Department of Statistics, North Carolina State
University, 2311 Stinson Drive, Campus Box 8203, Raleigh, NC 27695, U.S.A.\\
E-mail: \texttt{aschwar@ncsu.edu}\\
URL: \texttt{http://www4.stat.ncsu.edu/\~{}schwartzman/}

\end{small}
\end{quote}

\end{document}